\documentclass[11pt, final]{amsart}
\usepackage{amsmath,amsthm,amssymb}
\usepackage{graphicx}
\usepackage{color}

\usepackage{bbm}
\usepackage[all]{xy}
\usepackage{amscd}
\usepackage{stmaryrd}
\usepackage{verbatim}

\usepackage[english]{babel}
\usepackage{todonotes}

\theoremstyle{theorem}
\newtheorem{thm}{Theorem}[section]
\newtheorem{lem}[thm]{Lemma}
\newtheorem{cor}[thm]{Corollary}

\theoremstyle{remark}

\theoremstyle{definition}

\numberwithin{equation}{section}

\def\P{{\mathbb{P}}}
\def\R{{\mathbb{R}}}

\newcommand{\ind}{\mathbbm{1}}
\newcommand{\EE}{\mathbb{E}}
\newcommand{\PP}{\mathbb{P}}
\newcommand{\E}{\mathbb{E}}
\renewcommand{\P}{\mathbb{P}}
\newcommand{\NN}{\mathbb{N}}
\newcommand{\N}{\mathbb{N}}
\newcommand{\T}{\mathbb{T}}

\newcommand{\RR}{\mathbb{R}}
\newcommand{\Fcal}{\mathcal{F}}

\newcommand{\eps}{\varepsilon}

\newcommand{\ud}{\mathrm{d}}
\newcommand{\8}{\infty}
\renewcommand{\a}{\alpha}

\newcommand{\Ps}{\PP_{\alpha}^{\uparrow}}

\usepackage{soul}

\begin{document}

	\title{Local fluctuations of critical Mandelbrot cascades}

	\author[D. Buraczewski, P. Dyszewski and K. Kolesko]{Dariusz Buraczewski, Piotr Dyszewski and Konrad Kolesko}

\address{Dariusz Buraczewski, Mathematical Institute, University of
Wroc{\l}aw, Plac Grunwaldzki 2/4, 50-384 Wroc{\l}aw, Poland}
	\email{dbura@math.uni.wroc.pl}

\address{Piotr Dyszewski, Mathematical Institute, University of
Wroc{\l}aw, Plac Grunwaldzki 2/4, 50-384 Wroc{\l}aw, Poland}
	\email{pdysz@math.uni.wroc.pl}

\address{Konrad Kolesko, Mathematical Institute, University of
Wroc{\l}aw, Plac Grunwaldzki 2/4, 50-384 Wroc{\l}aw, Poland \newline 
Institut f\"ur Mathematik, Universit\"at Innsbruck, 6060 Innsbruck, Austria. }
	\email{kolesko@math.uni.wroc.pl}

	\thanks{The research was partially supported by the National Science Center, Poland (Sonata Bis, grant number DEC-2014/14/E/ST1/00588)}


	\begin{abstract}
		We investigate so-called generalized Mandelbrot cascades at the freezing (critical) temperature. It is known that, after a proper rescaling,
		a~sequence of multiplicative cascades converges weakly to some continuous random measure.  Our main question
		is how  the limiting measure $\mu$ fluctuates. For any given point $x$, denoting by  $B_n(x)$  the ball of radius $2^{-n}$ centered around $x$,
		we present optimal lower and upper estimates of  $\mu(B_n(x))$ as $n \to \infty$.
	\end{abstract}

	\maketitle

\section{Introduction}\subsection{Mandelbrot cascades}
\!\!\footnote{In order to keep the introduction as accessible as possible, we omit some technical subtleties. This way we focus on the general overview and the contribution of this paper.
We  postpone the proper introduction of the setting and notation to the second section.}
In the seventies Mandelbrot~\cite{mandelbrot1972, mandelbrot1999} proposed a model of~random multiplicative cascade measures, to simulate the energy dissipation in intermittent turbulence. Mandelbrot cascades exhibited a number of fractal  and statistical features observed experimentally in a~turbulence flow. Up to now, through various applications, this model  found its way into a~wide range of scientific fields from financial mathematics \cite{borland2005} to  quantum gravity and disordered systems in mathematical physics \cite{Barral:jin}.  Mathematically, a multiplicative cascade, is a~measure-valued stochastic process and was first rigorously described  by Kahane and Peyri\`{e}re~\cite{kahane:peyriere:1976}. They presented a complete proof of results announced by Mandelbrot, answering e.g. the questions of non-degeneracy, existence of moments and local properties. Since then multiplicative cascades become a subject of study for  numerous mathematicians, see e.g. \cite{barral:mandelbrot:1,barral:mandelbrot:2,BDGM,durrett:liggett,Liu2000}.

One of the simplest examples of  multiplicative cascades can be expressed as a sequence of random measures  on the unit interval $I=[0,1)$. They depend on two parameters: a~real number
$\beta>0$ (inverse temperature parameter) and a~real valued random variable $\xi$ (fluctuations). For convenience, we assume that $\xi$ is  normalized, i.e.
\begin{equation}\label{eq: a}
	\E e^{\xi} = \frac 12 \qquad \mbox{ and }\qquad \E \big[\xi e^{\xi}\big] = 0.
\end{equation}
To define the cascade measures, consider  an~infinite dyadic Ulam-Harris tree denoted by $\mathbb{T}_2 = \bigcup_{n \geq 0} \{0,1\}^n$ and attach to every edge, connecting $x$ with $xk$ ($x\in \mathbb{T}_2$, $k\in\{0,1\}$), a~random weight $\xi_k(x)$, being an independent copy of $\xi$. Let $V(x)$ be the total weight of the branch from the root to $x$ obtained by
adding weights of the edges along this path. Define the measure $\mu_{\beta,n}$ on the unit interval $I$  as an absolutely continuous with respect to the Lebesgue measure with
Radon-Nikodnym derivative constant on the set $I_x$,  such that the measure of set $I_x$ is equal to
\begin{equation*}
	\mu_{\beta, n}(I_x) = e^{-\beta V(x)},
\end{equation*}
where by $I_x$ we denote the dyadic interval coded by $x =(x_1,\ldots,x_n)\in \mathbb{T}_2$ such that $|x|=n$, i.e. $I_x = \big[\sum_{k=1}^nx_k 2^{-k},\sum_{k=1}^nx_k 2^{-k} + 2^{-n}\big)$.

After a normalization by proper a~deterministic sequence, say $c_{\beta,n}$, one obtains  measures $c_{\beta,n}\mu_{\beta,n}$ converging towards a finite nonzero random measure
$\mu_{\beta}$ on $I$. Essentially, due to self-similarity of the model,  asymptotic behavior of $\mu_{\beta,n}$ boils down to the asymptotic behavior of  its total mass, i.e.
\begin{equation*}
	Z_{\beta,n}=\mu_{\beta,n}(I) = \sum_{|x|=n}e^{-\beta V(x)}.
\end{equation*}

Derrida and Spohn~\cite{derrida1988} explained that behavior of the cascade  depends mainly on the~parameter $\beta$ and that there is a~phase transition in the behavior of the limiting measure. Under \eqref{eq: a} the critical value of parameter $\beta$ is 1. For $\beta < 1$ (high temperature) and $\beta = 1$ (freezing temperature) the limiting measure $\mu_{\beta}$
is continuous, although singular with respect to the Lebesgue measure, whereas for $\beta>1$ (low temperature) is purely atomic. In~the continuous case one of the fundamental problems is description of local behavior of the measure $\mu_{\beta}$, e.g. fluctuations of $\mu_{\beta}$, which is the main problem considered in this paper. More precisely we aim to find optimal,
deterministic functions $\phi_1$ and $\phi_2$ such that for $\mu_{\beta}$-almost all $x \in I$ and for 	sufficiently large $n$
we have  almost surely (a.s.)
\begin{equation*}
	\phi_1(n) \leq \mu_{\beta}(B_n(x)) \leq \phi_2(n),
\end{equation*}
where $B_n(x)$ is the dyadic set of length $2^{-n}$ containing $x$.

\subsection{The subcritical case}

If $\beta<1$ we say that the system is in the subcritical case or high
temperature case. In this setting, result of
Kahane and Peyri\`{e}re \cite{kahane:peyriere:1976}	 ensures that under some mild integrability assumptions
\begin{equation*}
	\big(\E Z_{\beta,n})^{-1} Z_{\beta,n} \to Z_{\beta} \quad \mbox{a.s.}
\end{equation*}
where $Z_{\beta}$ is a.s. positive and finite. Therefore one may infer, that for any fixed $x \in \mathbb{T}_2$, as $n \to \infty$
\begin{equation*}
	\big(\E Z_{\beta,n})^{-1} \mu_{\beta,n}(I_x) \to \mu_{\beta}(I_x) \qquad \mbox{a.s.}
\end{equation*}
the details are given in Section 2. Local fluctuations of $\mu_{\beta}$ were described by  Liu~\cite{Liu2000}, who proved that for any $\varepsilon>0$ for $\mu_{\beta}$-almost any $x\in I$ and for sufficiently large $n$
\begin{equation*}
	e^{-(a+\varepsilon)n } \leq \mu_{\beta}(B_n(x)) \leq e^{-(a-\varepsilon)n },
\end{equation*}
for some constants	 $a>0$ depending on $\beta$ and $\xi$. Liu~\cite{Liu2000} proved this estimates for generalized model, defined in Section 2.

\subsection{The critical case}

Whenever $\beta=1$, we say that the system is in the critical case, or~in the  freezing temperature. This is the main case considered in this article. The situation is more involved than in the subcritical case since in the latter, the limit $Z_{\beta}$ emerged as  that of
the~ positive martingale $(\E Z_{\beta,n})^{-1} Z_{\beta,n}$. Thus the proper choice of normalizing $c_{\beta,n}$ was natural. It turns out that in the critical case, this limit vanishes, showing that a~different scaling is needed in order to obtain a~nontrivial limit.
The~solution to this problem was recently delivered by A\"id\'ekon and Shi~\cite{aidekon2014} yielding
\begin{equation*}
	\sqrt{n} Z_{1,n} \to Z_{1} \quad \mbox{in probability}
\end{equation*}
with a.s. finite and positive $Z_{1}$. This convergence cannot be improved to a.s. convergence, since $\limsup_{n \to \infty}\sqrt{n}Z_{1,n}=\infty$ a.s. as is also proved in~\cite{aidekon2014}. From the convergence in probability however, we obtain that, as $n \to \infty$
\begin{equation*}
	\sqrt{n}\mu_{1,n}(I_x) \to \mu_{1}(I_x) \qquad \mbox{in probability.}
\end{equation*}
Barral~et~al.~\cite{barral2014} proved that $\mu_1$ is atomless and considered the problem of fluctuations. Under an additional assumption that $\xi$ is  Gaussian, it was proved that for certain $c>0$ and arbitrary $k>0$, with probability one for $\mu_{1}$-almost any
$x\in I$, for sufficiently large $n$
\begin{equation*}
	e^{-c\sqrt{n\log n}} \leq \mu_{1}(B_n(x)) \leq e^{-k\log n}.
\end{equation*}
As the main results of this article shows, these bounds are not optimal and can be improved to the estimates of the form
\begin{equation*}
	e^{-d\sqrt{n\log\log n}} \leq \mu_{1}(B_n(x)) \leq e^{-\sqrt{n}L(n)},
\end{equation*}
for some slowly varying function $L$ with $L(n)\to 0$ as $n \to \infty$. Moreover these estimates are valid for general, not necessary Gaussian, random variable $\xi$. Our results show also that these bounds are precise and give a detailed description
of the lower and upper time-space envelope of $\log\mu_{1}(B_n(x))$. Roughly speaking, we will show that it satisfies the same bounds as the sequence $e^{-V(x_n)}$, where $x_n$ is  a vertex of $n$th generation chosen at random in a way that will allow us to describe the behavior of $V(x_n)$.   For details see the discussion in Section~\ref{sec: results}.

\subsection{The supercritical case}

Situation when $\beta>1$ is referred to as the supercritical case or ''glassy'' low temperature phase. In this case, the asymptotic behavior of $Z_{\beta,n}$ is determined by the minima of $V(x)$ for $|x|=n$. Using the work of A\"id\'ekon~\cite{aidekon2013}, giving the weak
convergence of $\min_{|x|=n}V(x)  - 3/2 \log n$, Madaule~\cite{madaule2015} was able to prove that, as~$n \to \infty$
\begin{equation*}
	n^{\frac{3\beta}{2}}Z_{\beta,n} \to Z_{\beta} \quad \mbox{in distribution.}
\end{equation*}	
Whence we may infer that for $\beta>1$
\begin{equation*}
	n^{\frac{3\beta}{2}} \mu_{\beta,n}(I_x) \to \mu_{\beta}(I_x) \qquad \mbox{in distribution.}
\end{equation*}
However, as mentioned before, the limiting measure $\mu_{\beta}$ is purely atomic, see Barral~et~al.~\cite{barral:rhodes:vargas}.

\subsection{Liouville quantum gravity and Liouville Brownian motion} Finally, let us also mention that there is a natural counterpart of the measures $\mu_{\beta}$ in a continuous setting. Roughly speaking, these are random measures on a given domain $D$ in $\R^d$ with the following formal definition
$$\mu_{\gamma}=e^{\gamma X(x)-\gamma^2/2\EE[X(x)^2]}\sigma(dx),$$
where $X$ is a centered Gaussian random field with the property that $\operatorname{Cov}(X(x),X(y))=-\log\|x-y\|+O(1)$, as $\| x-y\| \to 0$ and $\sigma$ is a given Radon measure. Since $X$  cannot be defined pointwise even the existence of such a measure is far from being trivial (cf. \cite{kahane1985chaos, Berestycki2015}). In the particular case, when $X$ is a Gaussian free field with appropriate boundary condition these measures are important objects in theoretical physic because of the connections with the theory of Liouville quantum gravity. It is also worth to point out that Liouville quantum gravity is conjecturally related to discrete and continuum random surfaces. Roughly speaking, it appears as the scaling limit of a large class of statistical physical models (see for instance \cite{Rhodes2015} for further details and references therein).

Recently, Berestycki \cite{berestycki2015diffusion}  and  independently Garban, Rhodes and Vargas \cite{LBM, Rhodes2015} have constructed a diffusion on $D$, called Liouville Brownian motion, that has $\mu$ as a stationary measure. This process is conjectured to be the scaling limit of random walk on large planar maps. Finally, let as also mention the relationship between Liouville Brownian motion and the decay of the measure $\mu$: if  ${\bf p}_t(x,y)$ is its  transition probability  then it is believed that
$${\bf p}_t(x,y){\sim}\frac{e^{-\mu(B(x,|x-y|))/t}}{t}\qquad\text{as }|x-y|\to0.$$
See \cite{Rhodes2015} for further discussion and details.

\section{Generalized Mandelbrot cascades and main results}\label{sec: results}

\subsection{Generalized Mandelbrot cascades}	

The main aim of this article is to study asymptotic properties of the limiting measure. Since only the values of the measure are of  interest to us, we can regard the cascades as measures on some abstract space. This leads to so-called generalized Mandelbrot cascades
defined in the next few paragraphs.\\

Consider a one-dimensional, discrete-time branching random walk  governed by a~point process $\Theta$. We start with single particle placed at the origin of the real line. At time $n=1$ this particle dies while giving birth to a random number of particles which
will now form the first generation. Their position with respect to the birth place is determined by the~point process $\Theta$. At time $n=2$ each particle of the first generation, independently from the rest, dies while giving birth to the particles from the second generation.
Their positions, with respect to the birth place are determined by the same distribution $\Theta$. The system goes on according to this rules. Obviously the number of particles in each generation forms a~Galton-Watson process. We  denote the corresponding random tree
rooted at $\emptyset$ by $\mathbb{T}\subseteq \mathbb{U}$, where
\begin{equation*}
	\mathbb{U} = \bigcup_{k \geq 0} \NN^k.
\end{equation*}
We  write $|x| = n$ if $ x \in \NN^n$,  that is if $x$ is a particle at $n$th generation. We  denote the~positions of particles of the $n$th generation as $(V(x) \: | \: |x|=n)$, and the whole process as $(V(x) \: | \: x \in \mathbb{T})$. This process is usually referred
to as a branching random walk. For any vertices $x,y \in \mathbb{T}$, by $\llbracket x , y \rrbracket$ we  denote the shortest path connecting $x$ and $y$. We can partially order $\mathbb{T}$ by letting $x \geq y$ if $y \in \llbracket \emptyset, x \rrbracket$, that is if
$y$ is an~ancestor of $x$.  Let $x \wedge y = \inf \{ x , y\}$ be the oldest common ancestor of $x$ and $y$. For $x\in \T$ we introduce the set  of children of $x$ by 
$$
	C(x)=\{y\geq x\: | \:|y|=|x|+1\}
$$ and the set  of siblings of $x$ by $$\Omega(x)=\{y\: | \: |y|=|x|,\ |x \wedge y|=|x|-1 \}.$$ Finally, $x_i$  denotes the vertex in $\llbracket \emptyset, x \rrbracket$ such that $|x_i| = i$. The~branching random walk   gives  rise to a random measure on the boundary of $\mathbb{T}$, i.e.
\begin{equation*}
	\partial \mathbb{T} = \{ \xi \in  \mathbb{I}| \xi_n \in \mathbb{T}, \: n \in \NN \},
\end{equation*}
where $\mathbb{I} = \NN^{\NN}$.  For $\xi \in \mathbb{I}$ denote the truncation $\xi_n = \xi_{|\{1, \ldots , n  \}}$ and write $\xi>x$ for $x \in \mathbb{T}$ whenever $\xi_n=x$ for some $n\in \NN$. Notice that $\partial\mathbb{T}$ forms an ultrametric space with
\begin{equation*}
	B(x) = \{ \xi \in \partial \mathbb{T} \: | \: \xi>x \}, \quad x \in \mathbb{T}
\end{equation*}
as its topological basis. This corresponds to the choice of $d_c(x,y) = c^{-|x\wedge y|}$ with $c>1$ as a metric on $\partial \mathbb{T}$ in which $B(x)$ is a~ball of radius $c^{-|x|}$.

Observe that in a very special case $\Theta = \delta_{\xi_0}+\delta_{\xi_1}$, where $\xi_0$ and $\xi_1$ are iid and satisfy~\eqref{eq: a} the model reduces to the Mandelbrot cascade defined in the Introduction. Indeed, every particle has exactly two children, i.e. ${\mathbb{T}} = {\mathbb{T}}_2$ and  the intervals $I_x\subseteq I$ correspond to the balls $B(x)\subseteq \partial \mathbb{T}$.
However, below we work in full generality, when the number of children, the corresponding tree  $\mathbb{T}$ and its boundary $\partial\mathbb{T}$ are random.

\subsection{Assumptions and basic properties}

In this paper we work under a standard assumption that the branching random walk is in the so-called boundary (or critical) case, that is
\begin{equation}\label{eq:boundarycase}
	\EE\bigg[\sum_{|x|=1}{e^{-V(x)}}\bigg]=1 \qquad \mbox{and}\qquad \EE\bigg[\sum_{|x|=1}{V(x)e^{-V(x)}}\bigg]=0
\end{equation}
which boils down to \eqref{eq: a} if each particle has exactly two children. Throughout the paper we  use the convention that $\sum_{\emptyset} = 0$.  We need also some additional integrability assumptions, that is
\begin{equation}\label{eq:sigma}
	\sigma^2=\EE\bigg[ \sum_{|x|=1} V(x)^2e^{-V(x)}\bigg] <\infty
\end{equation}
as well as for some $p >2$,
\begin{equation}\label{eq:integrability}
	\EE\left[L(\log^+L)^p \right]<\infty,
\end{equation}
where $L=\sum_{|x|=1}(1+V^+(x))e^{-V(x)}$. Most of the discussions in this paper  become trivial if the system dies out, that is if $\mathbb{T}$ is finite. For example, by our definition $\partial \mathbb{T} = \emptyset$ for finite $\mathbb{T}$. Whence we assume
that the~underlying Galton-Watson process is supercritical, i.~e. $\EE[\sum_{|x|=1}1]>1$, so that the system survives with positive probability. Notice that this also implies the branching random walk is not reduced to a classical random walk, more precisely that the number
of offspring $\# \Theta$ is bigger than $1$ with positive probability.  To~avoid the need of considering the degenerate case we introduce the conditional probability
$$
	\P^*[\;\cdot\;] = \P[\;\cdot\;|{\mbox{nonextinction}}].
$$
Our main results will be formulated in terms of the measure $\P^*$. We will now focus on the~definition of the~measures $\mu_n$ and $\mu$ starting with defining the~total mass of the former via
\begin{equation*}
	 \mu_n(\partial \mathbb{T})= \sum_{|x|=n}{e^{-V(x)}}.
\end{equation*}
It can be easily shown, that thanks to the first condition in \eqref{eq:boundarycase} this sequence forms a~nonnegative, mean one martingale with respect to $\Fcal_n = \sigma(V(x) \: | \: |x|\leq n)$ (called the~\textit{additive martingale}), and whence is convergent a.s. It turns out that our second assumption in \eqref{eq:boundarycase} implies that the corresponding limit is $0$ (see for example Biggins~\cite{biggins1977}). Nevertheless A\"id\'ekon and Shi~\cite{aidekon2014} proved that, under
hypotheses~\eqref{eq:boundarycase}, \eqref{eq:sigma} and \eqref{eq:integrability}, we have the convergence
\begin{equation} \label{eq: aidekon shi}
	\sqrt{n}\mu_{n}(\partial \mathbb{T}) \to \mu(\partial \mathbb{T})  \qquad \mbox{  in probability}.
\end{equation}
Moreover, $\PP^*[\mu(\partial \mathbb{T}) > 0]=1$. This result holds true and was proven under slightly weaker assumptions than~\eqref{eq:integrability}. Since our main result requires~\eqref{eq:integrability}, we will continue to invoke other results with slightly stronger conditions for readers convenience.

Similarly, to define $\mu_n(B(y))$ for $y \in \mathbb{T}$, just truncate the additive martingale to the subtree of all branches containing $y$, that is
\begin{equation*}
	\mu_{n}(B(y))=\sum_{|x|=n, x>y} e^{-V(x)} = e^{-V(y)} \sum_{|x|=n, x>y} e^{-(V(x)-V(y))}
\end{equation*}
and by another appeal to \eqref{eq: aidekon shi} we infer that, as $n \to \infty$
\begin{equation*}
	\sqrt{n}\sum_{|x|=n, x>y} e^{-(V(x)-V(y))} \to W_y \quad \mbox{in probability}
\end{equation*}
for some nonnegative $W_y$. Whence, by  defining $\mu(B(y))$ as the limit in probability of $\sqrt{n}\mu_n(B(y))$, we get
\begin{equation*}
	\mu(B(y))= e^{-V(y)}W_y.
\end{equation*}
It can be easily verified that almost surely
\begin{equation*}
	\mu(\partial \mathbb{T}) =\sum_{|x|=1} e^{-V(x)} W_x.
\end{equation*}
Analogously, for any $y \in \mathbb{T}$ and $n>|y|$
\begin{equation*}
	W_{y} =\sum_{|x|=n,\: x>y} e^{-(V(x)-V(y)) } W_x.
\end{equation*}
Note that this means exactly that
\begin{equation*}
	\mu\bigg( \bigcup_{|x|=n, x>y} B(x) \bigg)=\mu(B(y)) = \sum_{ |x|=n, x>y }\mu(B(x))
\end{equation*}
since $B(y) = \bigcup_{|x|=n, x>y}B(x)$. By the extension theorem $\mu$ can be uniquely extended to a~measure on $\partial \mathbb{T}$.

\subsection{Main results} \label{sec:mainresults}

For functions $f,g \colon \NN \to \RR$ we write $f(n) \leq_{i.o.} g(n)$ if $f(n) \leq g(n)$ for infinitely many $n$ and $f(n) \leq_{a.a.} g(n)$ if $f(n) \leq g(n) $ for all but finitely many $n$.
We want to find deterministic functions $\phi_1,\phi_2 \colon \NN \to \RR$ such that
\begin{equation*}
	\phi_1(n) \leq_{a.a.} \mu(B(\xi_n)) \leq_{a.a.} \phi_2(n)
\end{equation*}
$\PP^*$-almost surely for $\mu$-almost all $\xi \in \partial \mathbb{T}$.

Our first result describes the upper time-space envelope of $\mu(B(\xi_n))$.
\begin{thm}\label{thm:mainresult1}
	Assume \eqref{eq:boundarycase}, \eqref{eq:sigma} and \eqref{eq:integrability}. Let $\psi \in C^1(\RR_+)$ be decreasing such that
	$t^{1/2-\delta}\psi(t)$ is eventually increasing to $+\infty$ for some $ \delta> \frac 1p$. Then $\PP^*$-almost surely for
	$\mu$-almost all
	$\xi \in \partial \mathbb{T}$
	\begin{equation*}
		\mu(B(\xi_n)) \leq_{a.a.} e^{-\sqrt{n}\psi(n)}, \quad \mbox{if } \int^{\infty} \frac{\psi(t)}{t} \: \ud t <\infty
	\end{equation*}
	and	
	\begin{equation*}
		\mu(B(\xi_n)) \geq_{i.o.} e^{-\sqrt{n}\psi(n)}, \quad \mbox{if } \int^{\infty} \frac{\psi(t)}{t} \: \ud t =\infty.
	\end{equation*}	
\end{thm}

Note that this result gives necessary and sufficient conditions for $\mu(B(\xi_n)) \leq_{a.a.} e^{-\sqrt{n}\psi(n)}$ allowing to describe the upper time-space envelope of $\mu(B(\xi_n))$ with arbitrarily small gap. In order to illustrate this, define the functions
$\psi_{k}, \psi_{k}^{(\varepsilon)} \colon \NN \to \RR$ for $k \in \NN$ and $\varepsilon > 0$ by
\begin{equation*}
	\psi_{k}(t) = \left[ \prod_{i=1}^{k} \log_{(i)}(t) \right]^{-1}, \quad
	\psi_{k}^{(\varepsilon)}(t) = \psi_{k}(t) \log_{k}(t)^{-\varepsilon},
\end{equation*}
where $\log_{(i)}(t)$ stands for the $i$th iterate of $\log(t)$. We can deduce from Theorem~\ref{thm:mainresult1} that
\begin{equation*}
	e^{-\sqrt{n} \psi_{k}(n)} \leq_{i.o.} \mu(B(\xi_n)) \leq_{a.a.} e^{-\sqrt{n}\psi_{k}^{(\varepsilon)}(n)}.
\end{equation*}
Since the same inequalities would hold if $\psi_{k}^{(\varepsilon)}$ or $\psi_{k}$ was multiplied by an arbitrary positive constant, we deduce that
\begin{equation*}
	\liminf_{n \to \infty} \frac{-\log(\mu(B(\xi_n)))}{\sqrt{n} \psi^{(\varepsilon)}_{k}(n)} = \infty \quad \mbox{and} \quad
	\liminf_{n \to \infty} \frac{ - \log(\mu(B(\xi_n)))}{\sqrt{n} \psi_{k}(n)} = 0
\end{equation*}
$\PP^*$-almost surely for $\mu$-almost all $\xi$. In particular for any $k \in \NN$ one has
\begin{equation*}
	\liminf_{n \to \infty} \frac{\log(-\log(\mu(B(\xi_n)))) - \frac{1}{2}\log n + \sum_{j=2}^k{\log_{(j)}(n)}    }{ \log_{(k+1)}(n)} = -1.
\end{equation*}
Our second result describes the lower time-space envelope.

\begin{thm}\label{thm:mainresult2}
	Assume \eqref{eq:boundarycase}, \eqref{eq:sigma}, \eqref{eq:integrability} and $\EE\left[\sum_{|x|=1}|V(x)|^{3+\varepsilon}e^{-V(x)} \right]<\infty$ for some $\varepsilon >0$. Then for any $\delta>0$, $\PP^*$-almost surely for $\mu$-almost all $\xi \in \partial \mathbb{T}$,
	\begin{equation*}
		\mu(B(\xi_n)) \geq_{a.a.} e^{-(1+\delta)\sqrt{2\sigma^2 n \log\log n}},
	\end{equation*}
	where $\sigma^2$ is given by \eqref{eq:sigma} and 	
	\begin{equation*}
		\mu(B(\xi_n)) \leq_{i.o.} e^{-(1-\delta)\sqrt{2\sigma^2 n \log\log n}}.
	\end{equation*}
\end{thm}

From the above one gets instantly that
\begin{equation*}
	\limsup_{n \to \infty}\frac{-\log(\mu(B(\xi_n)))}{\sqrt{2\sigma^2 n \log\log n}} = 1 \mbox{ a.s.}
\end{equation*}

\subsection{Discussion of the results}
The problem of describing local fluctuations of the Mandelbrot cascades in the critical case was previously investigated by~Barral et al.~\cite{barral2014}. They considered the case when $\mathbb{T}$ is the binary tree $\mathbb{T}_2$ and the branching random walk is generated by the Gaussian distribution and proved that for any $\varepsilon >0$ and any $k\in \NN$
\begin{equation*}
	\exp\left\{-(1+\varepsilon)\sqrt{2\log(2)n\log n} \right\}\leq_{i.o.} \mu(B(\xi_n))  \leq_{a.a.} \exp\left\{-k\log n\right\}
\end{equation*}
and
\begin{equation*}
	\exp\left\{-\sqrt{6 \log(2)} \sqrt{n (\log n + (1/3 + \varepsilon) \log\log n} \right\} \leq_{a.a.} \mu(B(\xi_n))
\end{equation*}
$\PP^*$-almost surely for $\mu$-almost all $\xi \in \partial\mathbb{T}$. Notice that the bounds appearing in the first part of this result, have different asymptotic. Thus, this result does not give a detailed information about
the upper time-space envelope of $\mu(B(\xi_n))$.

To prove Theorems \ref{thm:mainresult1} and \ref{thm:mainresult2} we use the spinal decomposition and the change of measure, based on the work of  Biggins, Kyprianou \cite{biggins2004}, used for example by A\"id\'ekon, Shi~\cite{aidekon2014}. After the change of measure, along every spine $\{w_n\}$
( i.e.~a random element of $\partial\mathbb{T}$), in the new probability space, $V( w_n)$ behaves as a random walk conditioned to stay positive and its fluctuations were studied by Hambly et al.~\cite{Hambly2003327}.  This result provides description of  time-space envelopes with arbitrary small error.

The details are given in Sections \ref{sec: random walk} -- \ref{sec: proofs}. First, in Section \ref{sec: random walk}, we recall some basic properties of the random walk conditioned to stay positive. In Section \ref{sec: change of prob} we describe the change of the probability space and finally in Sections \ref{sec: lemmas} and \ref{sec: proofs} we give complete proof of our results.

\section{One-dimensional random walk}
\label{sec: random walk}

In this section we introduce a one-dimensional random walk associated with the branching random walk defined above. Next we define a random walk conditioned to stay above some level $-\a$ for $\a\ge 0$ and formulate its fundamental properties concerning fluctuations of
its paths. Those results will play a~crucial role in our arguments.

\subsection{An associated one-dimensional random walk}

Assumption \eqref{eq:boundarycase} allows us to introduce a random walk  $\{S_n\}$ with the distribution of increments given by
\begin{equation*}
	\EE[f(S_1)]=\EE \bigg[ \sum_{|x|=1} f(V(x))e^{-V(x)}  \bigg],
\end{equation*}
for any measurable $f \colon \RR \to [0,\infty)$. Then, since the increments $S_2-S_1,S_3-S_2,\ldots$ are independent copies of $S_1$, one can easily show, that for any $n \in \NN$ and measurable $g \colon \RR^n \to \RR$ we have
\begin{equation*}
	\EE[g(S_1, S_2, \ldots , S_n)] = \EE\bigg[\sum_{|x|=n}{g(V(x_1),V(x_2), \ldots , V(x_n))e^{-V(x)}  } \bigg].
\end{equation*}
Note that by \eqref{eq:boundarycase}
$$
	\EE[S_1]=\EE\bigg[\sum_{|x|=1}V(x) e^{-V(x)} \bigg]=0
$$
and thus the random walk $\{ S_n\}$ is centered and by \eqref{eq:sigma} has a finite variance
$$
	\EE\left[S_1^2\right]=\EE\bigg[\sum_{|x|=1}V(x)^2e^{-V(x)} \bigg] = \sigma^2 <\8.
$$	
	
\subsection{A conditioned random walk}
It turns out that in our considerations an important role will be played not by the random walk $\{S_n\}$, but by its trajectories conditioned to stay above $-\a$ for some $\a\ge 0$.  Bertoin and Doney \cite{bertoin1994} showed that for each $k \in \mathbb{N}$,  $A \in \sigma(S_j, j\leq k)$
the limiting probabilities
\begin{equation*}
	\lim_{n \to \infty} \PP[A| \tau_{\alpha}>n],
\end{equation*}
where $\tau_{\alpha}= \inf \{k \geq 1 : S_k < -\alpha \}$, are well defined and nontrivial. Their result is a discrete analogue of the relationship between the Brownian motion and the Bessel-3 process. It~turns out that the conditioned random walk forms a Markov chain.
Here we sketch the~arguments leading to a description of its transition probability.

Since the random walk  $\{S_n\}$ is centered, $\tau^+ = \inf \{k \geq 1 : S_k\geq 0 \}$ is finite a.s. If we put
\begin{equation}\label{eq:defr}
	R(u) = \EE\bigg[\sum_{j=0}^{\tau^+-1}\ind_{\{ S_j \geq -u \}} \bigg],
\end{equation}
we see that by the duality lemma, $R$ is the renewal function associated with the entrance to $(-\infty, 0)$ by the walk $S$. That is, $R$ can be written in the following fashion
\begin{equation*}
	R(u) = \sum_{k \geq 0} \PP[H_k \geq -u]
\end{equation*}
for $u \geq 0$, where $H_0> H_1> H_2> \ldots$ are strictly descending ladder heights of $\{S_n\}$, $H_k = S_{\tau_k^-}$ with $\tau_0^-=0$ and $\tau_{k+1}^- = \min\{ j \geq \tau_k^- | S_j < H_k \}$ for $k \geq 0$.
One can show that $\EE\left[S_1^2 \right] < \infty$ and $\EE[S_1]  = 0$ ensure $\EE[|H_1|] < \infty$ (see e.g. Feller \cite{feller2008introduction}, Theorem XVIII.5.1).

As a consequence of \eqref{eq:defr}, by conditioning on $S_1$, one gets the following identity
\begin{equation}\label{eq: ru}
	R(u) = \EE[R(S_1+u) \ind_{\{ S_1 \geq -u \}}] \quad \mbox{for } u \geq0.
\end{equation}
Thus $\ind_{\{\tau_\a > n\}}R(S_n+\a)$ is a martingale. The corresponding $h$-transform defines a Markov chain such that for any measurable subset $A$ of $\R^k$
\begin{equation}\label{eq: strzalka}
	\PP_{\alpha}^{\uparrow}\big[ (S_1,\ldots, S_k)\in A\big] = \frac{1}{R(\alpha)}\EE\left[ \ind_{\{(S_1,\ldots, S_k)\in A\}\cap \{\tau_{\alpha}>k \}} R(S_k+\alpha) \right].
\end{equation}
Then the transition probability of this Markov chain is given by
$$
	P_{\alpha}^{\uparrow} (x,dy) = \ind_{\{ y\ge -\a\}}\frac{R(y+\a)}{R(x+\a)}P(S_1+x\in dy), \quad x\ge -\a.
$$
The random process $\{S_n\}$ under the probability measure $\Ps$ is called the random walk conditioned to stay in $[-\alpha , \infty)$.

\subsection{Some properties of the renewal function $R$}

Here we collect some properties of the function $R$, following from the renewal theorem, that will be needed in next sections. The renewal theorem (see e.g. Feller \cite{feller2008introduction}) distinguishes between two cases, when the random walk
$\{S_n\}$ is nonarithmetic (i.e. it is not contained in any set of the form $a{\mathbb Z}$ for  positive $a$) and when it is arithmetic. In the first case the renewal theorem says that for every $h>0$ the limit
\begin{equation*}
	 \lim_{u \to \infty} R(u+h)-R(u) = h/\E|H_1|
\end{equation*}
exists and is finite. If the random walk is arithmetic, then the same limit exists but only for $h$ and $u$ being multiplies of $a$:
$$
	\lim_{n\to\infty} R(na + h) - R(na) = h/\E|H_1|.
$$
Below we treat both cases simultaneously since we need just some simple consequences of the results stated above. In both cases the following limit exists
\begin{equation}\label{eq:ren}
	c_0 = \lim_{u \to \infty} \frac{R(u)}{u}\;.
\end{equation}
Whence there are constants $c_2 >  c_1 > 0$ such that for any $u \geq 0$
\begin{equation}\label{eq:renineq}
	c_1 (1 + u) \leq R(u) \leq c_2 (1 + u).
\end{equation}
Moreover, there is a constant $c_3>0$ such that for every $u,x>0$
\begin{equation}\label{eq:rennnn}
	R(u+x) - R(u) \le c_3(1+x).
\end{equation}

\subsection{Some properties of the conditioned random walk}

Here we describe some properties of trajectories of the Markov chain $(\{S_n\}, \Ps)$ that will be needed in the proofs of our main results. Analogously to the Bessel-3 process, paths of the conditioned random walk stays in 'some neighborhood' of $n^{1/2}$.
The precise description of its fluctuations was provided by Hambly et al. \cite{Hambly2003327} and is stated in the next two lemmas.

\begin{lem}[Law of the iterated logarithm]\label{lem: lil}
	Suppose that for some $\delta > 0$, $\EE |S_1|^{3+\delta}< \infty$. Then
	\begin{equation*}
		\limsup_{n \to \infty}{ \frac{S_n}{\sqrt{2n\sigma^2 \log\log n}} }=1 \qquad  \Ps\ \mbox{a.s.}
	\end{equation*}
\end{lem}

\begin{lem}[Lower Space-Time Envelope]\label{lem:envelope}
	Suppose that $\EE [S_1^2] < \infty$ and $\psi$ is a function on $(0,\8)$ such that $\psi(t) \downarrow 0$ and $\sqrt{t} \psi(t) \uparrow \infty$ as $t\uparrow \8$. Then
	\begin{equation*}
		\liminf_{n \to \infty}{  \frac{S_n}{ \sqrt{n}\psi(n) }  } = \infty \mbox{ or } 0 \qquad  \Ps\ \mbox{a.s.}
	\end{equation*}
	accordingly as
	\begin{equation*}
		\int^{\infty}{\frac{\psi(t)}{t} \: \ud t  } < \infty \mbox{ or } = \infty.
	\end{equation*}
\end{lem}

We will need also two further auxiliary lemmas reflecting the fact that trajectories of the~conditioned random walk goes to $+\8$. The first lemma is due to Biggins \cite{biggins2003}.

\begin{lem} \label{lem: biggins}
	Fix $y\ge x\ge -\a$. Then
	$$
		\Ps\big[ \min_{n\ge 1} S_n>x  \big|S_0 = y\big] = \frac{{ R( y-x)}}{R(\a+y)}.
	$$
\end{lem}

The next lemma seems to be standard, however since we don't know any reference we provide a complete proof of it.

\begin{lem}\label{lem:conditionedkozlov}
	For fixed $x>0$ there is $c_4$ such  that $$\Ps\big[\min_{k\geq n}S_k\le x \big] \leq \frac{c_4\log n}{\sqrt{n}}, \qquad  n>1.$$
\end{lem}

\begin{proof}
	In the proof we need the local limit theorem for conditioned random walks due to Caravenna \cite{caravenna}. In our settings this result implies that for fixed $h>0$ there is  $c_5$ such that
	\begin{align}\label{eq:nierownosc_Stona}
		\sup_{r \ge -\a}\Ps\big[r\le S_n \le r+h \big]\le \frac{c_5}{ \sqrt  n},
	\end{align}
	for any $n\ge 1$. We write
	\begin{equation}\label{eq: local}
	\begin{split}
		\Ps&\big[\min_{k\geq n}  S_k\le x \big] \\&\le \Ps\big[S_n\le 2x \big]\!\! +
		\Ps\big[\min_{k\geq n}S_k\le x;S_n > nx \big]
		\!\!+ \Ps\big[\min_{k\geq n}S_k\le x; 2x <   S_n \le nx \big].
	\end{split}
	\end{equation}
	Thus we have to bound three expressions. The first term, by \eqref{eq:nierownosc_Stona}, can be bounded in the following 	
	fashion:
	\begin{align*}
 		\Ps\left[S_n\le  2x \right] \le \frac{ c_5 { (2x+\a)}}{ h\sqrt n}.
	\end{align*}
	For the second one we use Lemma \ref{lem: biggins}, inequality \eqref{eq:rennnn} and the lower bound in~\eqref{eq:renineq}
	\begin{align*}
		\Ps\big[\min_{k\geq n}S_k\le x;S_n > nx \big]
			&= \int_{nx}^\8 \Ps\big[\min_{k\geq n}S_k\le x|S_n = y]\Ps[S_n \in dy \big]\\
			&= \int_{nx}^\8 \Big(1-\Ps\big[\min_{k\geq n}S_k> x|S_n = y]\Big)\Ps[S_n \in dy \big]\\	
			&=\int_{nx}^\8 \bigg( 1-  \frac{{ R( y-x)}}{R(\alpha+y)}\bigg)\,\Ps[S_n \in dy \big]\\
			&=\int_{nx}^\8 \frac{R(\alpha+y)-{ R(y-x)}}{R(\alpha+y)}\,\Ps[S_n \in dy \big]\\
			&\le \frac{c_3}{c_1}\int_{nx}^\8 {  \frac{1+\alpha + x}{1+\alpha+y}}\,\Ps[S_n \in dy \big]\\
			&\le \frac{c_3}{c_1} \frac{{ 1+\alpha+ x}}{1+\alpha+nx}\,\Ps[S_n >nx \big]\\
			&\le \frac{c_6}{n}\, ,
	\end{align*}
	for some $c_6$. To estimate the last term in \eqref{eq: local} we apply, successively, Lemma~\ref{lem: biggins}, \eqref{eq:rennnn}, the lower bound in \eqref{eq:renineq} and \eqref{eq:nierownosc_Stona}
	\begin{align*}
		\Ps\big[\min_{k\geq n}S_k\le x;\; & 2x <   S_n \le nx \big]\\
			&= \int_{2x}^{nx} \Ps\big[\min_{k\geq n}S_k\le x |   S_n =y \big] \Ps\big[ S_n \in dy \big]\\
			&= \int_{2x}^{nx}\Big( 1- \Ps\big[\min_{k\geq n}S_k> x |   S_n =y \big]\Big) \Ps\big[S_n \in dy \big]\\
			&= \int_{2x}^{nx}  \frac{R(\alpha + y)-R( y-x)}{R(\alpha +y)} \Ps\big[S_n \in dy \big]\\
			&\le\frac{c_3}{c_1} \int_{2x}^{nx}  \frac{1+\alpha+x}{1+\alpha+y} \Ps\big[S_n \in dy \big]\\
			&\le\frac{c_3}{c_1} \sum_{0 \le i \le xn/h}  \frac{1+\alpha+x}{1+\alpha+2x+ih} \Ps\big[2x+ih\le S_n \le 2x+(i+1)h \big]\\
			&\le  \frac{c_7(1+ \log(n))}{\sqrt n}\,.
	\end{align*}
	This completes the proof.
\end{proof}

\section{Derivative martingale and change of probabilities}
\label{sec: change of prob}

\subsection{Derivative martingale}

To study local properties of the random measure $\mu$ it is  convenient to express it in terms of  another fundamental martingale associated with the~branching random walk, namely of the derivative martingale. It is defined as
\begin{equation*}
	D_{n} = \sum_{|x|=n} V(x) e^{-V(x)}.
\end{equation*}
Our assumption \eqref{eq:boundarycase} ensures that this formula defines a centered martingale, i.e. $\E D_{n} = 0$ for all $n\in \N$. Convergence of the derivative martingale was studied by Biggins and Kyprianou \cite{biggins2004}, who proved that under assumptions
\eqref{eq:boundarycase}, \eqref{eq:sigma} and \eqref{eq:integrability}
$$
	D_{n} \to D,\qquad \PP^*\mbox{ a.s.}
$$
and $D >0$, $\P^*$ a.s. A\"id\'ekon and Shi \cite{aidekon2014} were able to relate $D$ with the limit of the additive martingale, that is $\mu(\partial \mathbb{T})$  (see \eqref{eq: aidekon shi}) and proved that
\begin{equation*}
	\mu(\partial \mathbb{T})=\left(\frac{2}{\pi \sigma^2} \right)^{1/2} D,\qquad \P^* \mbox{ a.s.}
\end{equation*}
Similarly, starting the derivative martingale from any vertex $x\in \T$, that is considering
\begin{equation*}
	D_{x,n} = \sum_{\substack{ |y|=|x|+n\\ y>x}} (V(y) - V(x)) e^{-(V(y)-V(x)) },
\end{equation*}
gives a.s. limit $D_x = \lim_{n\to \infty} D_{x,n}$. Since $W_x=c_8D_x$ we get
\begin{equation} \label{eq: muss}
	\mu(B(x)) = c_8 e^{-V(x)} D_x , \qquad \P^* \mbox{ a.s,}
\end{equation}
where $c_8 = \left(\frac{2}{\pi \sigma^2} \right)^{1/2}$.

\subsection{Change of probabilities}\label{sec: prob}

Our argument for Theorems~\ref{thm:mainresult1} and \ref{thm:mainresult2} require a change of the probability space. This is a~standard approach in the theory of branching random walks. An appropriate change of the probability measure reduces the main problem to a question expressed in terms of a~random walk on the real line. We would like
to use the fact that $\{D_n\}$ is a~martingale and apply the Doob h-transform. Unfortunately the derivative martingale is not positive. To overcome this difficulty we follow the approach based on the truncated argument presented in Biggins, Kyprianou \cite{biggins2004}
and  A\"id\'ekon, Shi \cite{aidekon2014}. For any vertex $x\in \mathbb{T}$ put
\begin{equation*}
	\underline{V}(x) = \min_{ y \in \llbracket \emptyset,x\rrbracket}{V(y)}.
\end{equation*}
Define the truncated martingale as
\begin{equation}\label{eq:truncated}
	D_{n}^{(\alpha)} = \sum_{|x|=n}{R(V(x) + \alpha)e^{-V(x)}\ind_{ \{\underline{V}(x)\ge-\alpha\}} },
\end{equation}
where $R$ is given by~\eqref{eq:defr}. Because of \eqref{eq:boundarycase} and \eqref{eq:ren}, we expect that for large values of $\a$, $D_n^{(\a)}$ should be comparable with $D_n$. In next sections we describe how these martingales are related with each other in
terms of the cascade measures.

Assuming  \eqref{eq:boundarycase}, Biggins and Kyprianou \cite{biggins2004} proved that for any $\alpha \geq 0$, $\big\{ D_n^{(\alpha)}\big\}$ is a~nonnegative martingale with $\EE\big[D_n^{(\alpha)}\big]=R(\alpha)$. Using this fact we can define a probability
measure $\PP^{(\alpha)}$ via
\begin{equation*}
	\PP^{(\alpha)}|_{\Fcal_n} = \frac{D_{n}^{(\alpha)}}{R(\alpha)} \cdot \PP|_{\Fcal_n}
\end{equation*}
that is for any $A \in \Fcal_n$,
\begin{equation} \label{eq: ppa}
	\PP^{(\alpha)}[A]=\EE\bigg[\ind_A \frac{D_{n}^{(\alpha)}}{R(\alpha)} \bigg].
\end{equation}
Giving rigorous arguments requires $\PP^{(\alpha)}$ to be defined on the space of marked trees with distinguished rays, that is infinite lines of  descents starting from the root or simply the elements of $\partial \mathbb{T}$ (we refer to Neveu~\cite{neveu1986} for more details).
The distinguished ray is called the spine and will be denoted by $\{w_n\}$. To explain how it is chosen, let for any $u>\a$,  $\widehat{\Theta}^{(\alpha)}(u)$ denote a~point process whose distribution under $\PP$ is the law of $(u + V (x), |x| = 1)$ under
$\PP^{(\alpha + u)}$.  We start with a~single particle placed at the origin of the real line. Denote this particle by $w_0 = \emptyset$. At $n$th moment in time (for $n > 0$), each particle of generation $n$ dies and gives birth to point processes independently of each other:
the particle $w_{n}$ generates a point process distributed as $\widehat{\Theta}^{(\alpha)}\big(V\big(w_{n} \big)\big)$ whereas other particle say $x$, with $|x| = n$ and $x \neq w_{n}$ generates a point process distributed as $V(x) + \Theta$. Finally the particle
$w_{n+1}$ is chosen among the children $y$ of $w_n$ with probability proportional to $R(\a+ V (y))e^{-V (y)} \ind_{ \{  \underline{V} (y)\ge -\alpha \}} $. An induction argument proves
\begin{equation*}
	\PP^{(\alpha)} \big[ w_n=x\big|\Fcal_n\big]= \frac{R(V(x)+\alpha)e^{-V(x)} \ind_{\{ \underline{V}(x)\geq -\alpha \}} }{D_n^{(\alpha)}}\;.
\end{equation*}

Note that, formally, the measure defined above is different from the one defined by the equation \eqref{eq: ppa}. However, since  there is a natural projection from the space of marked trees with  distinguished rays to the space of marked trees and  \eqref{eq: ppa} defines marginal law of $\P^{(\a)}$ defined on space of marked trees with  distinguished rays, we feel free, by slight abuse of notation, to use the same  symbol for both measures.

\subsection{Spine and conditioned random walk}

Biggins and Kyprianou \cite{biggins2004} proved that the~positions of the particles obtained in the way described above have the same distribution as the branching random walk under $\PP^{(\alpha)}$. Moreover the process $\big\{V(w_n)\big\}$ under
$\PP^{(\alpha)}$, is distributed as the centered random walk $\{S_n\}$ conditioned to stay in $[-\alpha, \infty)$.

Since the truncated martingale~\eqref{eq:truncated} is positive, it has an a.s. limit
\begin{equation}\label{eq: da}
	D^{(\alpha)} = \lim_{n \to \infty} {D^{(\alpha)}_n}.
\end{equation}
It turns out (see Biggins and Kyprianou \cite{biggins2004}) that this convergence holds also in  mean. This implies in particular that $\PP^{(\alpha)}$ is absolutely continuous with respect to $\PP$ with density $D^{(\alpha)} $, that is for any $A \in \Fcal$
\begin{equation*}
	\PP^{(\alpha)}[A] = \EE \bigg[\ind_A \frac{D^{(\alpha)}}{ R(\alpha)} \bigg].
\end{equation*}

\subsection{Truncated cascades}

The truncated martingale $D^{(\alpha)}_n$ introduced above is a useful tool to provide a different construction of the  measure $\mu$. The idea is to define a  truncated version of Mandelbrot cascades that converge to some limit measure which with high probability, up to a multiplicative constant, coincides with $\mu$. The advantage of this approach is that it  allows us to prove the upper bound in Theorem \ref{thm:mainresult1} and deduce continuity of measure $\mu$.

For given $\alpha\ge0$ and any $x\in \mathbb{T}$ we consider the martingale
\begin{equation}\label{eq: dnmtg}
	D_{x,n}^{(\alpha)} = \sum_{\substack{|y|=|x|+n\\ y>x}} R(V(y) +\alpha) e^{-(V(y)-V(x)) }\ind_{\{\underline{V}^x(y)\ge-\alpha\}},
\end{equation}
where for $y>x$, $\underline{V}^x(y) = \min_{ z \in \llbracket x,y\rrbracket}{V(z)}$. As before  $D_{x,n}^{(\alpha)}$ converges  almost surely and in the mean to the limit $D_{x}^{(\alpha)}$. We may define now  the  measure $\mu^{(\alpha)}$ on $\partial \T$ by setting
\begin{equation}\label{eq: mua}	
	\mu^{(\alpha)}(B(x))=\ind_{\{\underline{V}(x)\ge-\alpha\}}e^{-V(x)}D_{x}^{(\alpha)}.
\end{equation}
Note that since $\mu_n(\partial \mathbb{T})= \sum_{|x|=n}e^{-V(x)} \to 0$, $\inf_{|x|=n}V(x) \to \infty$. Thus, by \eqref{eq:ren} we have
$$
	\mu^{(\alpha)}(B(x))=c_0e^{-V(x)}D_{x}
$$
on the set $\{\min_{x\in\T}V(x)>-\alpha\}$. Since the probability of the last event is  at least $1-c_9 e^{-\a}$ (cf. formula (2.2) in  \cite{bk14}) we have
\begin{align}\label{eq:10}
	\mu=\frac{c_0}{c_8}\mu^{(\alpha)},
\end{align}
 with probability at least $1-c_9 e^{-\a}$.

\subsection{ Reduction to the measure $\P^{(\a)}$}

Our main result concerns $\P^*$-a.s. fluctuations of  $\mu(B( \xi_n))$ along infinite path $\xi \in \partial \mathbb{T}$ that does not belong to the null set of a measure~$\mu$.
However, here we explain how to reduce the problem to the measure $\P^{(\a)}$.
Instead of sampling $\xi$ according to $\mu$ (or its normalized version) we will take it as the~random spine $\{w_n\}$.  We already know that for $\a \ge 0$ the process
$\big( \P^{(\a)}, \{V(w_n)\}  \big)$ behaves like a~conditioned random walk. To reduce the main problem to this setting, we need to know that $\mu$--a.e. element from $\partial{\mathbb T}$ belongs to the range of a~spine in  $(\P^{(\a)}, \{w_n\})$ for some parameter $\a$. In other words we need to prove that the range of the spines $\big( \P^{(\a)}, \{w_n\}  \big)$ is a~relatively big subset of $\big(\P,\partial{\mathbb T}\big)$. We start with the following Lemma.

\begin{lem}\label{lem:asympbiggins}
	Assume \eqref{eq:boundarycase} and \eqref{eq:integrability}. We have
	\begin{equation*}
		\PP^*\big[D^{(\alpha)}=0 \big] \le c_9 e^{-\a} \to 0 \qquad \mbox{as } \alpha \to \infty.
	\end{equation*}
\end{lem}

\begin{proof}
	By \eqref{eq:renineq}
	$$
		R(x+\alpha)\ge c_2(1+x+\alpha)\ge c_2x.
	$$
	Since $D>0$, $\P^*$-a.s., on the set  $\big\{\min_{x\in \T}V(x)\ge-\alpha\big\}$ we have $D^{(\alpha)}\ge c_2 D>0$. Therefore,
	$$
		\PP^*\big[D^{(\alpha)}=0 \big]\le \PP^*\big[\min_{x\in \T}V(x) <-\alpha \big]\le  c_9 e^{-\a} ,
	$$
	 (see inequality (2.2) in \cite{bk14}).
\end{proof}

The following Lemma (and its respectively analogues statements for $\le_{i.o.},\ \ge_{a.a.},\  \ge_{i.o.}$) reduces our main results to the measure $\P^{(\a)}$. Formally, there is a factor of $\frac{c_0}{c_8}$ appearing in the claim. However, it does not affect our main results. To see that this is in
fact the case, note that in Theorems~\ref{thm:mainresult1} and \ref{thm:mainresult2}, $\phi(n) = e^{-\sqrt{n}\psi(n)}$, where $\sqrt{n}\psi(n) \to \infty$. One can see that, by the form of the integral test of $\psi$ in Theorem~\ref{thm:mainresult1} and an explicit expression of $\psi$ in Theorem~\ref{thm:mainresult2},
the factor $\frac{c_0}{c_8}$ can be in fact omitted in the statements of both Theorems.
\begin{lem}\label{lem:red} Suppose there are a constant $\a_0$ and a function $\phi$ such that for every $\a>\a_0$,
$\P^{(\a)}$ almost surely for $\mu^{(\a)}$ almost all $\xi \in {\mathbb T}$
$$
\mu^{(\a)}(B(\xi_n))\le_{a.a.}\phi(n).
$$
Then,
$\P^*$ almost surely for $\mu$ almost all $\xi \in {\mathbb T}$
$$
\mu(B(\xi_n))\le_{a.a.}\frac{c_0}{c_8}\phi(n).
$$
\end{lem}
\begin{proof}
To establish the Lemma we need to consider measures $\P^*$ and $\P^{(\a)}$ on the set of labelled rooted trees with distinguished rays i.e.
$$
	\mathcal X=\{(\bold t,\xi):\bold t \text{ is a labelled rooted tree, } \xi\in \partial \mathbb{T}\}.
$$
A labelled tree ${\bold t}$ is a pair $(\mathbb{T}, (A_x)_{x \in \mathbb{T}\setminus \{\emptyset\}})$, $\mathbb{T}$ is a rooted tree and $A_x$'s are real numbers representing the displacement of particle $v$ from its parent, i.e. $A_x = V(x) - V(x_{|x|-1})$.
Recall that for $\P^*$ almost all labelled rooted trees $\bold t$ we have a measure $\mu$ defined on $\partial \mathbb{T}$. It allows us to define a measure $\mathbb Q$ on $\mathcal X$ (with a canonical $\sigma$ algebra)  by
$$
	\mathbb Q(d\bold t, d\xi)=\P^*(d\bold t)\frac{\mu(d\xi)}{\mu(\partial \mathbb{T})}.
$$
Choose $\a>0$ and define
\begin{align*}
	\mathcal U&=\{(\bold t, \xi)\in \mathcal X:\mu(B(\xi_n))\le_{a.a.}\frac{c_0}{c_8}\phi(n)\},\\
	\mathcal {U}_\a&=\{(\bold t, \xi)\in \mathcal X:\mu^{(\a)}(B(\xi_n))\le_{a.a.}\phi(n)\}.
\end{align*}
We need to justify $\mathbb{Q}(\mathcal{U})=1$, knowing $\P^{(\a)}(\mathcal{U_\a})=1$ for $\alpha$ big enough. Take
$$
\mathcal A_{\alpha}=\{(\bold t,\xi)\in \mathcal X:\min_{x\in \mathbb{T}}V(x)>-\alpha\}.
$$ Since $\mu = \frac{c_0}{c_8} \mu^{(\a)}$ on the set $\mathcal A_{\alpha}$ and thus $\mathcal {U}^c \cap
\mathcal {A}_\a =  \mathcal {U}^c_\a \cap \mathcal {A}_\a$, for any $\delta>0$ we have
\begin{align*}
\mathbb Q(\mathcal U^c) &\le \mathbb Q(\mathcal A_{\alpha}\cap\mathcal U^c) + \mathbb Q(\mathcal A_{\alpha}^c) \\
&\le \frac{R(\a)}{\delta} \P^{(\a)}(\mathcal U^c_\a) + \P^*(D^{(\a)}<\delta) + \mathbb{Q}(\mathcal A^c_\a)\\
&= \P^*(D^{(\a)} < \delta) + \P^* (\mathcal A^c_\a).
\end{align*}
Passing with $\delta$ to 0, in view of Lemma \ref{lem:asympbiggins} we obtain
$$
\mathbb Q(\mathcal U^c) \le c e^{-\a} + \P^*(\mathcal A^c_\a).
$$
When $\a$ tends to infinity, the last expression converges to 0, thus $\mathbb Q(\mathcal U^c) = 0$.

\end{proof}

Let us also emphasise that, given the random label tree, the normalised measure $\frac{\mu^{(\alpha)}}{\mu^{(\alpha)}(\partial \T)}$ is the law of the (infinite) spine $\omega\in \partial \T$. In other words, sampling a label tree with distinguished ray from $\P^{(\alpha)}$ is equivalent to first sampling a label tree from $D^{(\alpha)}\P$ and then sampling a ray  from $\mu^{(\alpha)}$ (after normalisation). 

From now,
the main idea is use formula \eqref{eq: muss} and to show that the growth of $D_{w_n}$ does not interfere in the  behaviour of $\mu(B(w_n))$ that should be governed by $e^{-V(w_n)}$. However, under the changed measure $\P^{(\alpha)}$ the law of  $D_{w_n}$ depends on $V(w_n)$:  conditioned on $V(x)$ and $V(w_k)$ for $|x|,k\leq n$,   $D_{w_n}$ under $\P^{(\alpha)}$ has the same law as $D$ under $\P^{(\alpha+V(w_n))}$ and  it can be easily seen that the sequence is not even tight. This problem is the most significant difference between the critical and
subcritical case and in order to overcome this issue we need a convenient representation of $\mu(B(w_n))$, which is available in a slightly changed settings, as explained in the previous subsection (see \eqref{eq: mua}).

\section{ Some  auxiliary lemmas}
\label{sec: lemmas}

In this section we are going to prove some further properties of $\mu$ that will be needed in the proofs of Theorems \ref{thm:mainresult1} and \ref{thm:mainresult2}. Our main aim is to obtain a formula for the measure $\mu$ more suitable for our needs. This will be done in Lemma \ref{eq: reps}.
Reduction to the measure $\P^{(\a)}$ allows us to control the measure of a ball near a typical point and, in particular, prove its continuity needed in  Lemma \ref{eq: reps}. The fact that $\mu$ is continuous, under stronger assumptions, has been already shown in \cite{barral2014}, however here we work under much weaker moment hypotheses and cannot apply directly that results.

\subsection{ Continuity of $\mu$ and upper estimates}

The goal of this subsection is to establish that if $\int^{\infty} \psi(t)t^{-1}  \: \ud t < \infty$, then
\begin{equation*}
	\mu\big( B\big(w_n \big) \big) \leq_{a.a.} e^{-\sqrt{n}\psi(n)} \qquad \PP^{(\alpha)} \mbox{-a.s.}
\end{equation*}
This will in particular imply that $\mu$ is continuous. As explained above we deduce this result from analogous properties of the measure $\mu^{(\a)}$ considered with respect to $\P^{(\a)}$ for arbitrary large $\a$. Therefore  it follows immediately from the lemma:
\begin{lem}\label{lem: muuuu} Under hypotheses of Theorem \ref{thm:mainresult1}
		\begin{equation*}
			\mu^{(\a)}\big( B\big(w_n \big) \big) \leq_{a.a.} e^{-\sqrt{n}\psi(n)} \qquad \PP^{(\alpha)} \mbox{-a.s.}
		\end{equation*}
\end{lem}

\begin{proof}
Observe that
\begin{align}\nonumber
\mu^{(\a)}\big( B(w_n) \big)&\le e^{-V(w_n)} \sup_k D^{(\a)}_{w_n,k}\\
\label{eq:truncated_upper}
&\le \sum_{j\ge n} R(V(w_j)+\a) e^{- V(w_j)} \sup_k\sum_{x\in \Omega(w_{j+1})}e^{V(w_j)-V(x)} \frac{ D^{(\a)}_{x,k}}{R(V(w_j)+\a)}.
\end{align}

	{\bf Step 1.} First we prove that the contribution of  the second sum above  is negligible. For this purpose we show that for $\delta>1/p$ (for $p$ defined in \eqref{eq:integrability}) we have
	\begin{equation}\label{eq: upper}
		G_j^{(\a)}:=  \sup_k\sum_{x\in \Omega(w_{j+1})}e^{V(w_j)-V(x)} \frac{ D^{(\a)}_{x,k}}{R(V(w_j)+\a)} \leq_{a.a.} e^{j^{\delta}} \qquad \PP^{(\alpha)}\mbox{-- a.s.}
	\end{equation}
	To prove this inequality first we show that for $\beta>0$ and $F(\beta):=\EE^{(\beta)}\big[\big(\log^+( G_0^{(\beta)} ) \big)^{p} \big] $  we have
	\begin{equation}\label{eq:supdfin}
		\sup_{\beta>0}F(\beta)=c_{10} < \infty.
	\end{equation}
	By~\eqref{eq:renineq} we can take constant $c_{11}>0$ such that for all $\beta>0$  and $x\ge -\beta$ we have
	\begin{equation*}
		\frac{R(\beta +x)}{R(\beta)}\leq c_{11}(1+x^+).
	\end{equation*}
	Then, by the fact that under $\P^{(\beta)}$ conditioned on $x\in \Omega(w_1)$ the processes $(V(y)-V(x) \: | \: y \in \mathbb{T}_x)$ evolves independently of other branches of the process  and have the same law as $(V(y) \: | \: y \in \mathbb{T})$ {under $\PP^*$},  we can write 
	 \begin{equation*}
		\EE^{(\beta)}\left[\left(\log^+\big( G_0^{(\beta)}\big)\right)^{p}\right] 
 			=\EE^{(\beta)}\left[\sum_{|y|=1} \PP^{(\beta)}[w_1=y| \mathcal{F}_1] \: \EE^{(\beta)}\left[\left.  \left(\log^+\big( G_0^{(\beta)}\big)\right)^{p}\right| \mathcal{F}_1, w_1=y \right]  \right].
	\end{equation*}
	Due to the structure of $\PP^{(\beta)}$  we have
	\begin{equation*}
		\EE^{(\beta)}\left[\left.  \left(\log^+\big( G_0^{(\beta)}\big)\right)^{p}\right| \mathcal{F}_1, w_1=y \right]   = \EE^{(\beta)}\left[\left.  \left(\log^+\big(   \sup_k\sum_{x\in \Omega(y)}e^{-V(x)} \frac{ D^{(\a)}_{x,k}}{R(\a)}   \big)\right)^{p}\right| \mathcal{F}_1\right]
	\end{equation*}
	which leads us to 
		\begin{align*}
		&\EE^{(\beta)}\left[\left(\log^+\big( G_0^{(\beta)}\big)\right)^{p}\right]\\
		&=\EE^{(\beta)}\bigg[\sum_{|y|=1} \frac{R(V(y) +\beta)}{D_1^{(\beta)}} e^{-V(y)} \ind_{\{ {V}(y) \geq -\beta \}}
		\Big(\log^+\Big(\sup_{k\ge 0}\sum_{x\in \Omega(y)} \frac{ D^{(\a)}_{x,k}}{R(\beta)}e^{-V(x)}\ind_{\{ {V}(x) \geq -\beta \}}\Big)\Big)^p   \bigg]  \\
		&=\EE\bigg[\sum_{|y|=1} \frac{R(V(y) +\beta)}{R(\beta)} e^{-V(y)} \ind_{\{ {V}(y) \geq -\beta \}}
		\Big(\log^+\Big(\sup_{k\ge 0}\sum_{x\in \Omega(y)} \frac{ D^{(\a)}_{x,k}}{R(\beta)}e^{-V(x)}\ind_{\{ {V}(x) \geq -\beta \}}\Big)\Big)^p\bigg]  \\
		&\le\EE\bigg[\sum_{|y|=1} \frac{R(V(y) +\beta)}{R(\beta)} e^{-V(y)} \ind_{\{ {V}(y) \geq -\beta \}}
		\Big(\log^+\Big(\sup_{k\ge0}\sum_{|x|=1} \frac{ D^{(\a)}_{x,k}}{R(\beta)}e^{-V(x)}\ind_{\{ {V}(x) \geq -\beta \}}\Big)\Big)^p\bigg]  \\
			& =  \EE\bigg[\Big(\log^+\Big(\sup_{n\ge1}D^{(\beta)}_n/R(\beta)\Big)\Big)^{p} \sum_{|y|=1} \frac{R(V(y) +\beta)}{R(\beta)} e^{-V(y)} \ind_{\{ {V}(y) \geq -\beta \}} \bigg]\\
			 &\le  (2p)^pc_{11} \EE\bigg[\Big(\log^+\big(\sup_nD^{(\beta)}_n/R(\beta)\big)^{1/2p}\Big)^{p} \sum_{|y|=1}
					(1+V^+(y))e^{-V(y)}  \bigg].
	\end{align*}
	For the latter we can use the simple inequality
	$$ab\le e^a+b\log^+b,$$ valid for any $a,b\ge0$. Taking $$a=\Big(\log^+ \Big(\sup_nD^{(\beta)}_n/R(\beta)\Big)\Big)^{1/2p}\qquad  \mbox{ and } \qquad b^p=L=\sum_{|y|=1}(1+V^+(y))e^{-V(y)}$$ the
	above expectation can be bounded in the following way
	\begin{align*}
		\EE^{(\beta)}\left[\big(\log^+\left(G^{(\beta)}_{0}\right) \big)^{p} \right]
			&\le \EE\Big[\Big(\big(\sup_nD^{(\beta)}_n/R(\beta)\big)^{1/2p}+1+\frac1p L^{1/p}\log^+L\Big)^{p}\Big]\\
			 &\le3^p\EE\Big[\Big(\sup_nD^{(\beta)}_n/R(\beta)\Big)^{1/2}+1+ L\left(\log^+L\right)^{p}\Big]<\infty,
	\end{align*}
	since by Doob's martingale  inequality $\EE\left[\left(\sup_nD^{(\beta)}_n/R(\beta)\right)^{1/2} \right] \le 2$. This proves~\eqref{eq:supdfin}. Then for any $\alpha \geq 0$ and
	$n\in \NN$ we have
	\begin{align*}
		\PP^{(\alpha)} \left[  G_n^{(\a)} > e^{n^{\delta}} \right] & =
		\EE^{(\alpha)} \left[ \PP^{(\alpha)} \left[\left.  G_n^{(\a)} > e^{n^{\delta}}\right| V\left(w_{n}\right)\right]\right] \\ & =
		\EE^{(\alpha)} \left[ \PP^{(\alpha)} \left[\left. \left(\log^+\left( G_n^{(\a)}\right)\right)^p > {n^{p\delta}}\right| V\left(w_{n}\right)\right]\right]\\ & \le
		\EE^{(\alpha)} \left[ F(\alpha + V(w_{n})) {n^{-p\delta}}\right] \leq c_{10}n^{-p\delta}.
	\end{align*}
	The claim follows by the Borel-Cantelli lemma.

\medskip

	{\bf Step 2.} Now we prove the required upper bound. Recall that $\big(\P^{(\a)}, \{V(w_n)\}\big)$ has the same distribution as the conditioned random walk $\big(\Ps, S_n\big)$. We bound the $\mu^{(\a)}$ using Lemma~\ref{lem:envelope}
	 and \eqref{eq: upper}. Let us take any $\psi$  like in Theorem \ref{thm:mainresult1} then
	$$
		V(w_n) >_{a.a.} 8\sqrt n \psi(n)\qquad \P^{(\a)} \mbox{ a.s.}
	$$
	and hence, by  \eqref{eq:truncated_upper} and \eqref{eq: upper}
		\begin{align*}
			\mu^{(\alpha)}\big( B(w_n) \big) &\leq_{a.a.}\sum_{j\geq n}R(V(w_j)+\a) e^{- V(w_j)} \cdot G_j^{(\a)}\\
			&\leq_{a.a.} \sum_{j\geq n} e^{-4\sqrt{j}\psi(j)+j^{\delta}} \qquad \PP^{(\alpha)}\mbox{-a.s.}
		\end{align*}
		Using the monotonicity of $t^{\tfrac12-\delta}\psi(t)$ and $\psi(t)$ for sufficiently large $n$ we have
		\begin{equation*}
			\sum_{j\geq n} e^{-4\sqrt{j}\psi(j)+j^{\delta}} \leq \sum_{j \geq n}e^{-3\sqrt{j}\psi(j)}
			\leq \int_{n-1}^{\infty} e^{-3\sqrt{t}\psi(t)} \: \ud t\leq \int_{n}^{\infty} e^{-2\sqrt{t}\psi(t)} \: \ud t.
		\end{equation*}
		Again, since $\frac{\ud}{\ud t} \Big(t^{\tfrac12-\delta}\psi(t) \Big) \geq 0$, we infer that
		\begin{equation*}
			\frac{\psi(t)}{2\sqrt{t}} + \sqrt{t}\psi'(t) \geq \delta\frac{\psi(t)}{\sqrt{t}}
		\end{equation*}
		Note that left hand side of above inequality is just $\frac{\ud}{\ud t}\sqrt{t}\psi(t)$. Changing the variables in the integral to
		$s = \sqrt{t} \psi(t)$ gives
		\begin{eqnarray*}
			\int_{n}^{\infty} e^{-2\sqrt{t}\psi(t)} \: \ud t & \leq &
			\int_{n}^{\infty} e^{-2\sqrt{t}\psi(t)}\frac{\sqrt{t}}{\delta\psi(t)}\left(\frac{\psi(t)}{2\sqrt{t}}+\sqrt{t}\psi'(t) \right) \:\ud t \\& \leq &
			\int_{n}^{\infty} e^{-\sqrt{t}\psi(t)} \left( \frac{\psi(t)}{2\sqrt{t}} + \sqrt{t}\psi'(t) \right) \: \ud t =	
			\int_{\sqrt{n}\psi(n)} e^{-s} \: \ud s \\
			& = &e^{-\sqrt{n}\psi(n)}
		\end{eqnarray*}
		for $n \in \NN$ large enough. Whence
		\begin{equation*}
			\mu^{(\alpha)}\big( B(w_n ) \big) \leq_{a.a.} e^{-\sqrt{n}\psi(n)} \qquad \PP^{(\alpha)} \mbox{-a.s.}
		\end{equation*}

\end{proof}

As an immediate consequence of the Lemmas~\ref{lem: muuuu} and \ref{lem:red} we obtain upper estimates in Theorem \ref{thm:mainresult1} and continuity of $\mu$.
\begin{cor}\label{cor: up}
	Under hypotheses of Theorem \ref{thm:mainresult1} the measure $\mu$ is continuous and
	\begin{equation*}
		\mu\big( B\big(w_n \big) \big) \leq_{a.a.} e^{-\sqrt{n}\psi(n)} \qquad \PP^{(\alpha)} \mbox{-a.s.}
	\end{equation*}
	if $\int^{\infty} \psi(t) t^{-1} \: \ud t < \infty$.
\end{cor}

\subsection{A useful formula for the measure $\mu$}

In the proof of our main result the following representation of the sequence $\mu(B(w_n))$ will be needed:
\begin{lem}\label{eq: reps}
	For $\P^{(\alpha)}$ a.e. infinite ray $\{w_n\}\in \partial {\mathbb T}$ we have
	\begin{equation}\label{eq:dec}
		\mu\big(B(w_n ) \big) = \sum_{k \geq n }e^{-V(w_k)} \widehat{D}_k ,
	\end{equation}
	where
	\begin{equation*}
		\widehat{D}_n = c_8\sum_{x\in \Omega(w_{n+1})} e^{-\big(V(x) - V(w_n)\big)} D_x.
	\end{equation*}
	
\end{lem}
\begin{proof}[Proof of Lemma~\ref{eq: reps}]
	Because of \eqref{eq: muss}, we have 	
	\begin{align*}
		\mu\big(B(w_n) \big) & =  \sum_{x\in C(w_n)} \mu(B(x))\\
		&= \mu\big(B(w_{n+1}) \big) + \sum_{  x\in \Omega(w_{n+1})}\mu(B(x)) \\
		& =  \mu\big(B(w_{n+1}) \big) + c_8e^{-V(w_n)}\sum_{  x\in \Omega(w_{n+1})} e^{-\left(V(x) - V\left(w_n\right)\right)} D_x\\
		& =  \mu\big(B(w_{n+1}) \big) + e^{-V(w_n)} \widehat{D}_n.
	\end{align*}
	Notice that by iterating the formula above and $\mu(B(w_n)) \to 0$,  by continuity of $\mu$, we conclude the lemma.
	\end{proof}

We close this section with two more lemmas which establish that the contribution of $\widehat{D}_n$ is negligible by providing upper and lower estimates respectively.
\begin{lem}\label{lem:upperd}
	Assume \eqref{eq:boundarycase} and \eqref{eq:integrability}. Then for $\delta>1/p$
	\begin{equation*}
		\widehat{D}_n \leq_{a.a.} e^{n^{\delta}} \qquad \PP^{(\alpha)}\mbox{-- a.s.}
	\end{equation*}
\end{lem}
\begin{proof}

The proof of the lemma is similar to the first step of the proof of Lemma \ref{lem: muuuu}. For $\beta>0$ we set $F(\beta)=\EE^{(\beta)}\Big[\big(\log^+(\widehat{D}_0)\big)^p \Big]$ and we have
 \begin{align*}
		F(\beta)&=\EE\bigg[\sum_{|y|=1} \frac{R(V(y) +\beta)}{R(\beta)} e^{-V(y)} \ind_{\{ {V}(y) \geq -\beta \}}
		\Big(\log^+\Big(\sum_{x\in \Omega(y)} D_xe^{-V(x)}\Big)\Big)^p\bigg]  \\
		&\le\EE\bigg[\sum_{|y|=1} \frac{R(V(y) +\beta)}{R(\beta)} e^{-V(y)} \ind_{\{ {V}(y) \geq -\beta \}}
		\Big(\log^+\Big(\sum_{|x|=1} D_xe^{-V(x)}\Big)\Big)^p\bigg]  \\
		&=\EE\bigg[\sum_{|y|=1} \frac{R(V(y) +\beta)}{R(\beta)} e^{-V(y)} \ind_{\{ {V}(y) \geq -\beta \}}
		\big(\log^+(D_{\emptyset})\big)^p\bigg] \\
		 &\le  (2p)^pc_{11} \EE\bigg[\Big(\log^+\big(D^{\tfrac1{2p}}\big)\Big)^{p} \sum_{|y|=1}
		(1+V^+(y))e^{-V(y)}  \bigg]<\infty.
	\end{align*}
	For $t>2$ and $\gamma=\log t$ we have
	\begin{align*}
	\PP[D>t ]&\le\PP\Big[D>t,\ \min_{x\in T}V(x)>\gamma\Big]+\PP[\min_{x\in \T}V(x)\le\gamma]\\
	&\le \PP\Big[c_0^{-1}D^{(\gamma)}>t]+e^{-\gamma}\le c_0^{-1}t^{-1}R(\gamma)+e^{-\gamma}\le c_{12}t^{-1}\log t,
	\end{align*}
	 for some constant $c_{12}$. In particular $\EE\left[{D}^{1/2} \right] < \infty$ which in turn, by the same argument as in the proof of Lemma \ref{lem: muuuu}, implies that the function $F$ is bounded by some constant $c_{13}$.

	Finally, for 	$n\in \NN$ we have
	\begin{align*}
		\PP^{(\alpha)} \left[ \widehat{D}_n > e^{n^{\delta}} \right] & =
		\EE^{(\alpha)} \left[ \PP^{(\alpha)} \left[\left.  \widehat{D}_n > e^{n^{\delta}}\right| V\left(w_{n}\right)\right]\right] \\ & =
		\EE^{(\alpha)} \left[ \PP^{(\alpha)} \left[\left. \left(\log^+\big( \widehat{D}_n\big)\right)^p > {n^{p\delta}}\right| V\left(w_{n}\right)\right]\right]\\ & \le
		\EE^{(\alpha)} \left[ F(\alpha + V(w_{n})) {n^{-p\delta}}\right] \leq c_{13}n^{-p\delta}.
	\end{align*}
	The claim follows by the Borel-Cantelli lemma.

\end{proof}

\begin{lem}\label{lem:lowerd}
	Assume \eqref{eq:boundarycase}. There exists $\eta> 0$ such that for all sufficiently large $\alpha \geq 0$
	\begin{equation*}
		\max_{n^3\leq j < (n+1)^3 }\widehat{D}_j \geq_{a.a.} \eta \qquad \PP^{(\alpha)}\mbox{-- a.s.}
	\end{equation*}
\end{lem}

\begin{proof}
	By our assumptions we can infer existence of $M$, $\delta_0$, $\delta_1$, $\delta_2 >0$ such that
	\begin{itemize}
		\item $\PP\left[ \mbox{there are $x\not=y$ such that $|x|=|y|=1$ and } V(x),V(y) \in (-M,M) \right] \geq \delta_0$
		\item $\PP[D > \delta_1] > \delta_2$
	\end{itemize}
	Note that for $\alpha > 2M$, by \eqref{eq:renineq}, we have
	\begin{equation*}
		\frac{R(\alpha-M)}{R(\alpha)} \geq \frac{c_1}{2c_2} =: \delta_3.
	\end{equation*}
	The following claim holds true with $\eta = \delta_1e^{-M}$, $\delta = \delta_0\delta_2\delta_3 e^{-M}$
	\begin{equation}\label{eq: st}
		\PP^{(\alpha)} \big[\widehat{D}_0> \eta \big] \geq \delta
	\end{equation}
	for any $\alpha > 2M$. Indeed,
	\begin{align*}
		\PP^{(\alpha)}\big[ \widehat{D}_0 > \eta \big] & =  \PP^{(\alpha)}\bigg[\sum_{\substack{|x|=1,\\w_{1}\neq x}} e^{-V(x)} D_x > \eta \bigg]\\
			& \geq \PP^{(\alpha)} \left[ e^{-V(x)} D_x > \eta \mbox{ for some } x \neq w_{1}, \: |x|=1 \right] \\
			 & \geq  \PP^{(\alpha)} \left[ e^{-V(x)} > e^{-M}, \: D_x> \delta_1  \mbox{ for some } x \neq w_{1}, \: |x|=1 \right] \\
			& =  \EE^{(\alpha)}\left[\PP^{(\alpha)} \left[\left.  V(x)< M, \: D_x> \delta_1 \mbox{ for some } x \neq w_{1}, \: |x|=1 \right| \Fcal_1  \right] \right] \\
			 & \geq  \EE^{(\alpha)}\left[ \ind_{\left\{ V(x)< M \mbox{ \tiny for some } x \neq w_{1}, \: |x|=1 \right\}} \PP\left[ D > \delta_1 \right]\right] \\
			 & \geq  \delta_2\PP^{(\alpha)} \left[V(x)< M, \mbox{ for some } x \neq w_{1}, \: |x|=1 \right].
	\end{align*}
	The remaining probability can be bounded from below by
	\begin{align*}
		\PP^{(\alpha)}\bigg[\bigcup_{\substack{|x|=1,\\w_{1}\neq x}} V(x)< M  \bigg] & =  \EE \bigg[ \sum_{|y|=1} \ind_{\left\{ \exists x \neq y \: V(x) < M \right\}} \frac{R(\alpha + V(y))}{R(\alpha)}  e^{-V(y)} \ind_{\{{V}(y) \geq -\alpha  \}} \bigg] \\
				& \geq  \delta_3\EE \bigg[ \sum_{|y|=1} \ind_{\left\{\exists x \neq y \: V(x)<M \right\}} e^{-V(y)}\ind_{\{V(y)\in (-M, M)\}}\bigg] \\
				& \geq  \delta_3e^{-M}\PP\left[ \exists x \neq y \: V(x),V(y) \in (-M,M) \right]\\  &\geq \delta_0\delta_3e^{-M}.
	\end{align*}
	This proves \eqref{eq: st}.

\medskip

	At the end of the proof of this lemma, we will invoke the Borel-Cantelli lemma, so first we consider the sequence
	\begin{align*}
		\PP^{(\alpha)} \left[\min_{k^3 \leq i < (k+1)^3} \widehat{D}_i < \eta \right]  & \leq   \PP^{(\alpha)} \left[ \min_{k^3\leq i < (k+1)^3} V\big( w_i\big) <M \right] \\
				&  +\; \PP^{(\alpha)} \left[\max_{k^3 \leq i < (k+1)^3} \widehat{D}_i < \eta; \: \min_{k^3\leq i < (k+1)^3} V\big( w_i\big) \geq M \right].
	\end{align*}
	From here, we will use an induction argument in order to show that for any $k,l \in \NN$ and $\a>2M$
	\begin{equation}\label{eq:indmaxd}
		\PP^{(\alpha)} \left[\max_{k \leq i < l} \widehat{D}_i < \eta; \: \min_{k\leq i < l} V\big( w_i\big) \geq M \right] \leq (1-\delta)^{l-k+1}.
	\end{equation}
	Put
	\begin{equation*}
		A(k,l)=\left\{\max_{k \leq i < l} \widehat{D}_i < \eta; \: \min_{k\leq i < l} V\left( w_i\right) \geq M \right\}.
	\end{equation*}
	For $l =k$ \eqref{eq:indmaxd} follows from \eqref{eq: st}. For bigger $l$ we argue that
	\begin{align*}
		\PP^{(\alpha)} [A(k,l+1)]& =  \EE^{(\alpha)} \left[ \PP^{(\alpha)} \left[A(k,l+1)\left| V\big(w_{i}\big), \widehat{D}_i, i\leq l, V\big(w_{l+1}\big) \right. \right]  \right] \\
			& =  \EE^{(\alpha)} \left[\ind_{A(k,l)}\ind_{\left\{V\big(w_{l+1}\big) >M \right\} } \PP^{(\alpha)} \left[\left. \widehat{D}_{l+1} \leq \eta \right|  V\left(w_{l+1}\right)\right]  \right] \\
			& =  \EE^{(\alpha)} \left[ \ind_{A(k,l)}\ind_{\left\{V\left(w_{l+1}\right) >M \right\} } \PP^{\left(\alpha+V\left(w_{l+1}\right)\right)} \left[\widehat{D}_{1} \leq \eta \right] \right] \\
			& \leq  \PP^{(\alpha)} \left[ A(k,l) \right] (1-\delta).
	\end{align*}
	Finally, from \eqref{eq:indmaxd} and Lemma~\ref{lem:conditionedkozlov} we infer that for sufficiently large $k$
	\begin{equation*}
		\PP^{(\alpha)} \left[  \max_{k^3 \leq j < (k+1)^3}\widehat{D}_j \leq \eta \right] \leq (1-\delta)^{3k^2}+C\log k / k^{3/2}.
	\end{equation*}
	From this the claim follows by Borel-Cantelli lemma.
\end{proof}

\section{Proofs of Theorems \ref{thm:mainresult1} and \ref{thm:mainresult2}} \label{sec: proofs}
	
\begin{proof}[Proof of Theorem~\ref{thm:mainresult1}] \noindent
	{\bf Step 1. Upper bound.} The upper bound was already proved in Corollary \ref{cor: up}

\medskip

	\noindent
	{\bf  Step 2. Lower bound.}
	In view of \eqref{eq:dec}
	$$
		\mu\big(B(w_n)\big) \ge e^{-V(w_n)}\widehat D_n.
	$$
	Lemma \ref{lem:envelope} ensures that for any $\psi$ such that $\int^{\8} \psi(t)dt/t = \8$ we have
	\begin{equation} \label{eq: vwn}
		V(w_n)<_{i.o.} \frac 12 \sqrt n\psi(n)\qquad \P^{(\a)} \mbox{ a.s.,}
	\end{equation}
	that implies
	\begin{equation}\label{eq: evv}
		e^{-V(w_n)}>_{i.o.} e^{-\frac 12 \sqrt n\psi(n)}\qquad \P^{(\a)} \mbox{ a.s.}
	\end{equation}
	The idea of the proof is the following. We will prove that for some large constant $M$, $\widehat D_n > e^{-M}$, $\P^{(\a)}$ i.o., moreover this estimate holds i.o. on the set where
	\eqref{eq: evv} is satisfied. For a rigorous argument choose $\eps>0$  and $M$ such that
	$$
		\P\big[{\text{there are }x\neq y, |x|=|y|=1 \text{ and }  -M<V(x),V(y)<M}\big] > \eps.
	$$
	and $$\P[D>1]\ge\eps.$$
	Next, we define a~sequence of stopping times
	$$
		T_0 = 0, \qquad T_k = \inf\big\{ n> T_{k-1}: 0<V(w_n)< \sqrt n\psi (n)/2 \big\},
	$$
	where the corresponding filtration is given by $\Fcal_n^* = \sigma(  \{V(x)  : x\in\T\setminus\T_{w_{n}}\cup \{w_n\}\})$, where
	$\mathbb{T}_{w_{n}} \subseteq \mathbb{T}$ denotes the tree rooted at $w_{n}$.
	Because of \eqref{eq: vwn}, these stopping times are finite a.s. For an event
	\begin{multline*}
		A_{k} = \big\{|{V}(x_0)-V(w_{T_k})|<M \mbox{ and } D_{x_0} > 1  \mbox{ for some } x_0\in C(w_{T_k})\setminus \{w_{T_k+1}\}\big\},
	\end{multline*}
	  observe  that $A_{k-1} \in \Fcal_{T_k}^*$. For any $\beta>0$ we set $$ H(\beta)=\PP^{(\beta)}[A_0]$$ and observe that for $\beta >2M$ and  $\delta=\frac{c_1}{2c_2}$ we have
	 \begin{align*}
		\PP^{(\beta)}[A_0] &\ge\EE^{(\beta)}[\PP^{(\a)}[\exists |x_0|=1,x_0\neq w_1,\: |{V}(x_0)|<M ]]\P[D>1]\\
		&\ge \eps\EE[\sum_{|x|=1}\frac{R(V(x)+\beta)}{R(\beta)}\ind_{\{V(x)+\beta\ge0\}}e^{-V(x)}\ind_{\{\exists_{x_0\neq x}  \: |V(x_0)|<M\}}]\\
		&\ge\eps\EE[\sum_{|x|=1}\delta\frac{1+V(x)+\beta}{1+\beta}\ind_{\{|V(x)|<M\}}e^{-M}\ind_{\{\exists_{x_0\neq x}  \: |V(x_0)|<M\}}]\\
		&\ge\eps\tfrac{\delta }{2}e^{-M}\PP[\exists{x_0\neq x, |x_0|=|x|=1},\:|V(x_0)|<M,|V(x)|<M]\\
		&\ge\tfrac{\delta }{2}e^{-M}\eps^2.
	\end{align*}
	Since for $\alpha>2M$
	\begin{align*}
		\PP^{(\alpha)}[A_k|\Fcal_{T_k}^*]=\EE^{(\a)}[ H(\alpha+V(w_{T_k})) |\Fcal_{T_k}^*]]\ge\tfrac{\delta }{2}e^{-M}\eps^2,
	\end{align*}
	 the conditioned Borel-Cantelli lemma (see e.g. Corollary 5.29 in Breiman \cite{Breiman}) $\P^{(\a)}[A_{k}\ \mbox{i.o.}]=1$. Hence
	$$
		\mu\big( B(w_n)\big)\ge_{i.o} c_8 e^{-\frac 12 \sqrt n \psi(n)} e^{-\alpha} e^{-M} \ge_{i.o} e^{-\sqrt n \psi (n)}\qquad \P^{(\a)} \mbox{ a.s.}
	$$
\end{proof}

\begin{proof}[Proof of Theorem \ref{thm:mainresult2}]

	{\bf Step 1. Lower bound.} As in the previous case it is sufficient to prove the result for $\P^{(\a)}$ a.e. spine $\{w_n\}$ and all sufficiently large $\a$. Take an arbitrary
		$\delta>0$. Then by Lemma \ref{lem: lil}
		\begin{equation*}
			V\big(w_n \big) \leq_{a.a.} (1+\delta/2) \sqrt{2\sigma^2n\log\log n }  \qquad \PP^{(\alpha)} \mbox{ a.s.}
		\end{equation*}
		From this we obtain
		\begin{equation*}
			\mu\big( B\big( w_n \big) \big) \geq_{a.a.}
			\sum_{j \geq n} e^{-(1+\delta/2)\sqrt{2\sigma^2j\log\log(j)}}\widehat{D}_j.
		\end{equation*}
		Now we use Lemma~\ref{lem:lowerd} and for the sequence $(k_n)_n$ such that $(k_n-1)^3 \leq n < k_n^3$ we write,
		since $\widehat{D}_j\ge 0$
		\begin{align*}
			\mu\big( B\big( w_n \big) \big)
				& \geq_{a.a.}  \bigg(\sum_{n\leq j < k_n^3} + \sum_{k_n^3\leq j \leq (k_n+1)^3} + \sum_{j > (k_n+1)^3} \bigg)
						e^{-(1+\delta/2)\sqrt{2\sigma^2j\log\log(j)}}\widehat{D}_j \\
				& \geq_{a.a.}  \sum_{k_n^3\leq j \leq (k_n+1)^3} e^{-(1+\delta/2)\sqrt{2\sigma^2j\log\log(j)}}\widehat{D}_j \\ & \geq_{a.a.}
						\eta e^{-(1+\delta/2)\sqrt{2\sigma^2(k_n+1)^3\log\log((k_n+1)^3)}}\\
				&\geq_{a.a.}  e^{-(1+\delta)^2\sqrt{2\sigma^2n\log\log n}}\;.
		\end{align*}
		This completes the first step.
\medskip

	\noindent
	{\bf Step 2. Upper bound.} Fix $\delta>0$. First we write { lower} estimates for the spine $\{V(w_k)\}$. Lemma \ref{lem: lil} gives
	\begin{equation}\label{eq:1}
		V(w_n) >_{i.o.} (1-\delta/8) \sqrt{2n \sigma^2 \log\log n}\qquad \P^{(\a)}\ \mbox{a.s.}
	\end{equation}
	Choosing $\psi(n) = 1/(\log n)^2$ in Lemma \ref{lem:envelope}  we obtain
	\begin{equation}\label{eq:2}
		V(w_n) >_{a.a.} \frac{3\sqrt n}{(\log n)^2}\qquad \P^{(\a)}\ \mbox{a.s.}
	\end{equation}
	We define a sequence of stopping times (this is a subsequence of the indices for which \eqref{eq:1} holds)
	\begin{align*}
		T_1 &= \inf\Big\{ n:\; V(w_n) > (1-\delta/8)\sqrt{2\sigma^2 n\log\log n}  \Big\},\\
		T_{k+1} &= \inf\Big\{ n \ge T_k(\log T_k)^5:\; V(w_n) > (1-\delta/8)\sqrt{2\sigma^2 n\log\log n}  \Big\}.\\
	\end{align*}
	In view of \eqref{eq:1} these stopping times are finite $\P^{(\a)}$ a.s. Denote
	$$
		A_{k+1} = \Big\{V(w_{n})\ge (1-\delta/4)\sqrt{2\sigma^2 T_k \log\log T_k}\ \mbox{ for }\ T_k < n \le T_k (\log T_k)^5\Big\}.
	$$
	We will prove  that $\P^{(\a)}[A_n \mbox{ i.o.}]=1$. Notice that, applying Lemma~\ref{lem: biggins} and \eqref{eq:renineq} we have
	\begin{align*}
		\P^{(\a)}\big[A_{k+1}| {\mathcal F}_{T_k}^*\big] &
				= \P^{(\a)}\Big[\min_{T_k < n \le T_k (\log T_k)^5} V(w_n) > (1-\delta/4)\sqrt{2\sigma^2 T_k \log\log T_k}\Big| V(w_{T_k})\Big]\\
				& \ge  \P^{(\a)}\Big[\min_{n> T_k  } V(w_n)>(1-\delta/4)\sqrt{2\sigma^2 T_k \log\log T_k}\Big| V (w_{T_k})\Big]\\
				&\ge \frac{{ R( \delta/8 \sqrt{2\sigma^2 T_k \log\log T_k})}}{R(\a + (1-\delta/8)\sqrt{2\sigma^2 T_k \log\log T_k})}
				> \frac{1}{1+\alpha}\delta.
	\end{align*}
	Since $A_k\in {\mathcal F}_{T_k}^*$,  the conditioned Borel-Cantelli lemma (Corollary 5.29 in Breiman \cite{Breiman}) implies that $\P^{(\a)}[A_k\mbox{ i.o.}]=1$. Therefore, by \eqref{eq:1}, \eqref{eq:2} and Lemma \ref{lem:upperd}, for $n = T_k$
	such that $A_k$ holds, we have
	\begin{align*}
		\mu\big(B(w_n) \big)  &\le_{i.o} \sum_{n\le k \le n(\log n)^5} e^{-(1-\delta/4)\sqrt{2\sigma^2 n\log\log n}} e^{k^{\delta}}+ \sum_{k > n(\log n)^5} e^{-\frac{3\sqrt k}{(\log k)^2}} e^{k^{\delta}}\\
			&\le_{i.o} n(\log n)^5 e^{-(1-\delta/4)\sqrt{2\sigma^2 n\log\log n}} e^{(n(\log n )^5)^{\delta}} + \sum_{k > n(\log n)^5} e^{-\frac{2\sqrt k}{(\log k)^2}} \\
			&\le_{i.o}  e^{-(1-\delta/2)\sqrt{2\sigma^2 n\log\log n}} + \bigg(\frac{2\sqrt{n(\log n)^5}}{(\log (n(\log n)^5))^2}\bigg)^4\times e^{-\frac{2\sqrt{n(\log n)^5}}{(\log (n(\log n)^5))^2}} \\
			&\le_{i.o}  e^{-(1-\delta)\sqrt{2\sigma^2 n\log\log n}}.
	\end{align*}
	This proves Theorem~\ref{thm:mainresult2}.
\end{proof}

\section*{Acknowledgment}

 The authors would like to thank two anonymous referees for their constructive suggestions that helped improving the presentation of this paper.

\bibliographystyle{plain}
\bibliography{bibliography}

\end{document}